\documentclass[reqno, a4paper, 12pt]{amsart}
\usepackage{amsfonts,amsthm,amsmath,amssymb,amscd,amsmath,epsf,latexsym,enumerate,mathrsfs}
\usepackage{graphicx}
\usepackage[utf8x]{inputenc}
\newtheorem{theorem}{Theorem}

\newtheorem{lemma}[theorem]{Lemma}

\newtheorem{example}{Example}
\newtheorem{remark}{Remark}
\newtheorem{problem}{Problem}
\newcommand{\bC}{\mathbb{C}}

\newtheorem{CONJ}{Conjecture}

\begin{document}
          \numberwithin{equation}{section}

          \title[]
          {NON-REAL ZEROS OF POLYNOMIALS  IN A POLYNOMIAL SEQUENCE SATISFYING A THREE-TERM RECURRENCE RELATION}

\author[I.~Ndikubwayo]{Innocent Ndikubwayo}
\address{Department of Mathematics, Stockholm University, SE-106 91
Stockholm,         Sweden}

\email {innocent@math.su.se, ndikubwayo@cns.mak.ac.ug}


\begin{abstract} This paper discusses the location of zeros of polynomials in a polynomial sequence $\{P_n(z)\}_{n=1}^{\infty}$ generated by a three-term recurrence relation of the form 
$P_n(z)+ B(z)P_{n-1}(z) +A(z) P_{n-k}(z)=0$ with $k>2$ and the standard initial conditions $P_{0}(z)=1, P_{-1}(z)=\ldots=P_{-k+1}(z)=0,$ where $A(z)$ and $B(z)$ are arbitrary coprime  real  polynomials. We show that there always exist  polynomials in $\{P_n(z)\}_{n=1}^{\infty}$ with non-real zeros. 
\end{abstract}
\maketitle
\section{Introduction}
For decades, a popular topic of studies in mathematics is related to three-term recurrence relations subject to natural restrictions on their coefficients. By Favard's theorem \cite{GRM}, such  recurrences generate orthogonal polynomials and these are of great interest since they are frequently used in many problems in the approximation theory, mathematical and numerical analysis, and their applications (for example, least square approximation of functions, difference and differential equations, Gaussian quadrature processes, etc.), see \cite{shai}.

In general, the zeros of polynomials $P_n(z)$ generated by recurrences do not exactly lie on a particular curve but are attracted to a curve (which in this paper we shall call the limiting curve) as $n \to \infty$. Such a limiting curve is explicitly described in \cite{BKW1, BKW2}. Recently, K. Tran in \cite{Tr1, TI} has proved cases where the  polynomials $P_n(z)$ generated by three-term recurrences have all their zeros  (for all  or  sufficiently large $n$) situated exactly on the said limiting curve.  We begin with the following conjecture.

\begin{CONJ}   {\cite[Conjecture]{Tr1} }   \label{conj:Tran}
For an arbitrary pair of polynomials $A(z)$ and $B(z)$, all  zeros of every polynomial in the sequence $\{P_n(z)\}_{n=1}^\infty$ satisfying the three-term recurrence relation of length $k$
\begin{equation}\label{eq:tran}
P_n(z)+B(z)P_{n-1}(z)+A(z)P_{n-k}(z)=0
\end{equation}
with the standard initial conditions $P_0(z)=1$, $P_{-1}(z)=\ldots=P_{-k+1}(z)=0$ which do satisfy $A(z)\neq 0$  
lie on the  algebraic  curve  $\Gamma \subset \bC$ given by 
\begin{equation}\label{eq:trcurve}
\Im \left(\frac{B^k(z)}{A(z)}\right)=0\quad {\rm and}\quad 0\le (-1)^k\Re \left(\frac{B^k(z)}{A(z)}\right)\le \frac{k^k}{(k-1)^{k-1}}.
\end{equation}
Moreover, these roots become dense in $\Gamma$ when $n\to \infty$.
\end{CONJ}
In the same paper, the above conjecture was proven for $k = 2,3,4.$ In \cite{TI}, K. Tran settled Conjecture $A$ for polynomials $P_n(z)$ with sufficiently large $n$. The problems around this area of study have most recently received substantial interest and a number studies have been carried out, see for example the papers \cite{TI,KD,BO, TrZu1,ndikus, TrZu2}. In \cite{BO}, the authors proved the following theorem. 

\begin{theorem}[see \cite{BO}]\label{th:main}
For an arbitrary pair of polynomials $A(z)$ and $B(z)$, all  the zeros of every polynomial in the sequence $\{P_n(z)\}_{n=1}^\infty$ satisfying the three-term recurrence relation of length $k $
\begin{equation*}\label{eq:general}
P_n(z)+B(z)P_{n-\ell}(z)+A(z)P_{n-k}(z)=0
\end{equation*}
where $k$ and $\ell$ are coprime and with the standard initial conditions $P_0(z)=1$, $P_{-1}(z)=\ldots=P_{-k+1}(z)=0$ which  satisfy the condition $A(z)B(z)\neq 0$  
lie on the real algebraic  curve $\mathcal{C}$  given by 
\begin{equation}\label{eq:generalcurve}
\Im \left(\frac{B^k(z)}{A^\ell(z)}\right)=0.
\end{equation}
\end{theorem}
The above theorem completely settles the first part of Conjecture~\ref{conj:Tran}.
There has been an initial attempt  to obtain the exact portion of the curve $\mathcal{C}$ where the zeros of the polynomials lie by providing in addition to (\ref{eq:generalcurve}), an  inequality constraint satisfied by the real part of the rational function $\frac{B^k(z)}{A^\ell(z)}$. This has  been proven for specific cases namely, $(k, \ell)= (3,2)$ and $(4,3)$ respectively and the details of the proofs can be found in \cite{ndikus}. In the same paper based on numerical experiments,  a more general conjecture for the real part of  $\frac{B^k(z)}{A^\ell(z)}$ has been proposed for this problem.

In the present paper, it is of interest to determine where in complex plane the zeros of every polynomial in the sequence $\{P_n(z)\}_{n=1}^{\infty}$ generated by (1.1) are located. In a particular case of $k=2$,  the author  in \cite{KD} characterizes real polynomials $A(z)$ and $B(z)$ to ensure that all the generated polynomials $P_n(z)$ are  hyperbolic.  This paper is a sequel of \cite{KD} but for $k>2$. We aim at proving whether or not it is possible to generalize the former.

\begin{problem}\label{prob:main}
{\rm} In the above notation, consider the recurrence relation \begin{equation}\label{christ}
P_n(z)+ B(z)P_{n-1}(z)+ A(z)P_{n-k}\\
(z)=0 \end{equation} where $k >2$ with the standard initial conditions, \begin{equation}
\label{kachap} P_{0}(z)=1, P_{-1}(z)=\ldots=P_{-k+1}(z)=0,\end{equation} where $A(z)$ and $B(z)$ are arbitrary real polynomials. Characterize  $A(z)$ and $B(z)$ (if possible) such that all the $P_{n}(z)$ are hyperbolic. 
\end{problem}

To formulate our main result, we need to look at the curve defined by the first condition in $(\ref{eq:trcurve}).$ We shall view  $\mathbb{C}P^1$ as $\mathbb{C}\cup \{\infty\}$, the extended complex plane and $\mathbb{R}P^1$ as  the extended real line.

Let $f:\mathbb{C}P^1 \to \mathbb{C}P^1$ be the rational function defined by $f(z)=\frac{B^k(z)}{A(z)}$ where $A(z)$ and $B(z)$ are real polynomials. Denote by $\widetilde{\Gamma} \subset \mathbb{C}P^1$ the curve given by $\Im(f(z))=0$, that is $\widetilde{\Gamma}=\{ z \in \mathbb{C}P^1: \Im(f(z))=0 \}=f^{-1}(\mathbb{R}P^1).$

 For real polynomials $A(z)$ and $B(z)$, define the curve $\Gamma$ by the condition \eqref{eq:trcurve}. It is clear that $\Gamma \subset \widetilde{\Gamma}.$

\medskip

In the remaining part of this section, let us remind the reader of some  basic definitions and facts about rational functions. For further details, see \cite{KD}.

 For a non-constant rational function $R(z) = \frac{P(z)}{Q(z)}$, where $P(z)$ and $Q(z)$ are polynomials with no common zeros, the degree of $R(z)$ is defined as the maximum of the degrees of $P(z)$ and  $Q(z)$. A  point $z_0 \in \mathbb{C}P^1$ is called a critical point of $R(z)$, (and $R(z_0)$ a critical value) if $R(z)$ fails to be injective in any neighbourhood of $z_0$, that is, either $R'(z_0)=0$ or $R'(z_0)=\infty$ (i.e, at the zeros of $Q(z)$). 
The order of a critical point $z_0$ of $R(z)$ is the order of zero of  $R'(z)$ at $z_0$.

Given a pair $(P(z), Q(z))$ of polynomials, we define their Wronskian as the polynomial $\mathcal{W}(P, Q):= P'Q-Q'P $  where $P'$ and $Q'$ are derivatives of $P$ and $Q$ with respect to $z$ respectively. If $P$ and $Q$ have no common zeros, then the zeros of $\mathcal{W}(P, Q)$ are exactly the critical points of the rational map $R(z)$.  In fact if $\alpha$ is a multiple zero of $R$, then $\alpha$ is a  zero of the Wronskian.
 
We call a non-zero univariate polynomial with real coefficients hyperbolic if all its zeros are real. In  {  \cite[\textsection  3.1]{JBPB}},  we find that the zeros of two hyperbolic polynomimals $P(z), Q(z) \in \mathbb{R}[z]$ interlace if and only if $|$deg $P$ - deg $Q| \leq 1$ and $\mathcal{W}(P, Q)$ is either nonnegative or nonpositive on the whole real axis. 
Notice that to say that the zeros of $P$ and $Q$  interlace means that each zero of $Q$ lies between two successive zeros of $P$ and there is at most one zero of $Q$ between any two successive zeros of $P$,  \cite{DKM}.
More information about the Wronskian can be found in \cite{KDI}.
\begin{remark} \label{in2}
For  the rational function $f(z)=\frac{B^k(z)}{A(z)}$, we have 
$$f'(z)= \frac{B^{k-1}(z)(kA(z)B'(z)-B(z)A'(z))}{A^2(z)}=\frac{\mathcal{W}(B^k(z), A(z))}{A^2(z)}.$$ We observe that the critical points of $f(z)$ are the zeros of the Wronskian $\mathcal{W}(B^k(z), A(z))$ or the poles of $f(z)$.
In particular, if  $A(z),B(z) \in \mathbb{R}[z]$ are coprime polynomials where $B(z)$ is hyperbolic with  distinct zeros, then all the zeros of $B(z)$ are  real critical points of  $f(z)$ each with multiplicity $k-1$.  
\end{remark}

Let $P(x)$ be a univariate polynomial of degree $n$ with zeros $x_1, \dots, x_n$ and leading coefficient $a_n$. The ordinary discriminant of $P(x)$ denoted by $\operatorname{Disc}_x(P(x))$ is  defined as  \begin{eqnarray*}
\operatorname{Disc}_x(P(x))=a_n^{2n-2} \prod_{1\leq i<j\leq n}(x_i-x_j)^2.
\end{eqnarray*}

Generally, the discriminant of a polynomial connects with the ratio of its zeros in the sense that the discriminant  is zero if and only if  the polynomial has multiple zeros. In particular, the discriminant of a polynomial vanishes  whenever there exist at least a zero with multiplicity greater or equal to $2$. For more details  on the ordinary discriminants, see \cite{Ge}.

\begin{example} \label{exxa}
For coprime $1 \leq \ell < k$, discriminant  of a trinomial 
 \begin{eqnarray*} 
 P(x)= a x^k + bx^\ell +c
 \end{eqnarray*}
is given by $k^kc^{k-1}a^{k-1}+ (-1)^{k-1}\ell^\ell(k-\ell)^{k-\ell}c^{\ell-1}b^ka^{k-\ell-1}. $ In particular,  \begin{eqnarray} \label{eqq}
\operatorname{Disc}_x(x^k+ Bx+A)=A^{k-2}\left(k^kA+(-1)^{k-1}(k-1)^{k-1}B^k\right).
\end{eqnarray}
\end{example}
The expression \eqref{eqq} will be of interest later in this work.

The main result of this paper is as follows.
\begin{theorem} {\rm}In the above notation of Problem \ref{prob:main} for $k>2$,  there always exist  polynomials in the sequence $\{P_n(z)\}_{n=1}^{\infty}$ with non-real zeros.
\end{theorem}

\section{Proofs}
\begin{lemma}\label{abov}
For $k >2$, consider the recurrence relation \begin{equation}\label{christi}
P_n(z)+ B(z)P_{n-1}(z)+ A(z)P_{n-k}\\
(z)=0 \end{equation} with the standard initial conditions, \begin{equation*} P_{0}(z)=1, P_{-1}(z)=\ldots=P_{-k+1}(z)=0,\end{equation*} where $A(z)$ and $B(z)$ are arbitrary  polynomials. If $f(z)= \frac{B^k(z)}{A(z)}$, then the zeros of $P_{1}(z)$ and $P_2(z)$ are critical points of $f(z)$.
\end{lemma}

\begin{proof}
Substitution of the initial conditions  in the recurrence relation
$$P_n(z)+ B(z)P_{n-1}(z) +A(z) P_{n-k}(z)=0$$ for $n=1$ gives  $P_1(z)=-B(z).$  Similarly since $k>2$, repeating the same process gives $P_2(z)=B^2(z).$ The zeros of $P_{1}(z)$ and $P_2(z)$ are the zeros $B(z).$ The conclusion that the zeros of $P_{1}(z)$ and $P_2(z)$ are critical points of $f(z)$ follows from  Remark \ref{in2}.
\end{proof}
\begin{remark}
For the recurrence \eqref{christi} to generate a sequence of hyperbolic polynomials, $B(z)$ must be  hyperbolic by Lemma \ref{abov}.
 \end{remark}
Let us briefly discuss some facts about the limiting curve $\Gamma$. Let 
$P(\lambda, z)= \lambda^k + B(z)\lambda + A(z)$ be the characteristic polynomial of the recurrence \eqref{christ} and $\lambda_1(z), \lambda_2(z), \ldots, \lambda_k(z)$ be its distinct non-zero characteristic roots. For $i \neq j$, let 
$\Gamma_{i,j}:= \{ \alpha \in \mathbb{C}: |\lambda_i(\alpha)|= |\lambda_j(\alpha)|\}$ be
the  equimodular curve of $P(\lambda, z)$ associated to the characteristic functions  $\lambda_i(z)$ and $\lambda_j(z)$. For a fixed $\alpha \in \mathbb{C}$ with $i\neq j$, we have $|\lambda_i(\alpha)|= |\lambda_j(\alpha)|$
if and only if there exists an $s \in \mathbb{C}$ such that $|s| = 1$ and $\lambda_i(\alpha)=s\lambda_j(\alpha).$ 

For each $i\neq j$, let $w=w(z):=\lambda_i(z)/\lambda_j(z)$    and define
 $P_w(\lambda,z)= P(w\lambda,z)=w^k\lambda^k + B(z)w\lambda + A(z)$. Then it is clear that $\lambda_j(z)$ is a common solution of both $P(\lambda,z) = 0$ and $P_w(\lambda,z)= 0$. 
A  necessary and sufficient condition for $P_w(\lambda,z)$ and $P(\lambda,z)$ to have a non-constant common factor is that their resultant $\rho(w, z)$ vanishes as a function of $z$.
By Lemma 3 \cite[\textsection  3]{Biggs},  $\rho(w, z)=A(z)(w-1)^k\Delta_k (w, z)$ where $\Delta_k (w, z)$ is a reciprocal polynomial in $w$ of degree  $ k(k - 1)$. In addition, $\Delta_k (1, z)$ is a multiple of Disc$_{\lambda}(P(\lambda, z))$, the discriminant of $P(\lambda, z)$.
In the same paper, it is proved that the reciprocal polynomial $\Delta_k (w, z)$ can be written as
 \begin{eqnarray*}\label{ecclipse}
\Delta_k (w,z)=w^{k(k-1)/2}v(t,z) \end{eqnarray*}
 where  $t = w + w^{-1} + 2$.
 The equimodularity condition $|w| = 1$ corresponds to  $t$ being in the real interval $0 \leq t \leq 4$. In particular, \begin{eqnarray}\label{kall}
 v(4,z)=\Delta_k (1,z)= \pm \operatorname{Disc}_{\lambda}(P(\lambda, z)). \end{eqnarray}

 \begin{lemma} \label{Namak}
 Let $A(z)$ and $B(z)$ be as defined in \eqref{christ} and  $R(t, z)= t^k + B(z)t + A(z)$. Further, let $U:=\{z: \operatorname{Disc}_t(R(t,z))= 0 \}$ and $\Gamma$ be the curve defined in \eqref{eq:trcurve}. A point  $z_0\in \mathbb{C}$ is an endpoint of $\Gamma$  if and only if $ z_0 \in U$  and $A(z_0)\neq 0$. 
\end{lemma}
\begin{proof}
Let $E= U \setminus\{z\in \mathbb{C}: A(z)= 0 \}$. We observe that $z_0 \in E$ if and only if
$\operatorname{Disc}_t(R(t,z_0))=0$ and $A(z_0)\neq 0$. This is equivalent to $v(4,z_0)=0$ and $A(z_0)\neq 0$ by \eqref{kall}. But by the results proved  in {\cite[\textsection  5] {Biggs}}, $v(4,z_0)=0$ and $A(z_0)\neq 0$ if and only if
 $z_0$ is an endpoint of segments of $\Gamma.$  The proof is complete.
 \end{proof}
 
\begin{remark}
For the recurrence \eqref{christi} to generate a sequence of hyperbolic polynomials, it is a  necessary condition that  the $z_0 \in \mathbb{C}$ mentioned in Lemma \ref{Namak} is real.
\end{remark}
We now state the following theorem about the behaviour of analytic functions near a critical point. This will be used in the proof of the main result. For details of the proof, see \cite{SBB}.

For $\delta >0$ and $z_0 \in \mathbb{C}$, we define
$D_\delta(z_0 ) := \{z \in \mathbb{C}:  |z − z_0 | < \delta\}$.
\begin{theorem} \label{TTH}
 Let $g(z)$ be  a non-constant analytic, function in a region, $\Omega \subset \mathbb{C}$. Let $z_0 \in \Omega$, 
$ w_0 = g(z_0)$, and suppose that $g(z)- w_0$ has a zero of order $p \geq 2$ at $z_0$. 
The following hold; 
\begin{enumerate}[(a)]\item
 There are $\epsilon, \delta > 0$ such that for every $w \in D_\epsilon(w_0) \setminus 
 \{w_0 \}$, there are exactly $p$ distinct solutions of 
 \begin{eqnarray} \label{innocencia}
 g(z) = w
 \end{eqnarray}
 with $z \in D_{\delta}(z_0)$. Moreover, for these solutions, $g(z)- w$ has a simple
zero.
\item There is an analytic function, $h$, on $D_{\epsilon^{1/p}}(0)$ with $h(0) = 0, h'(0)\neq 0$, so that if $w \in D_{\epsilon}(w_0)$ and
\begin{eqnarray*}
w = w_0 + \tau  e^{i\theta},~~~~~~ 0 < \tau < \epsilon, 0 \leq \theta < 2\pi 
\end{eqnarray*}
then the $p$ solutions of (\ref{innocencia}) are given by
\begin{eqnarray*}
z = z_0 + h(\tau^{1/p} e^{( i(\theta +2\pi j)/p )}), ~~~~j = 0,1,\dots, p-1. 
\end{eqnarray*}

\item There is a power series, $\sum_{n=1}^{\infty}b_n x^n$, with radius of convergence at least $\epsilon$, so the solutions of (\ref{innocencia}) are given by
$$z = z_0 + \sum_{n=1}^{\infty}b_n(w-w_0)^{n/p}$$
where $(w-w_0 )^{1/p}$ is interpreted as the pth roots of $(w-w_0 )$ (same root
taken in all terms of the power series).
\end{enumerate}
\end{theorem}
Let us finally settle the main result of this paper.

\begin{proof}[Proof of Theorem 2]
Suppose that $B(z)$ is hyperbolic, otherwise the theorem follows from Lemma \ref{abov}.
Additionally, let $z_0$ be a zero of $B(z)$ with multiplicity $p>0$. Then $z_0$ is a zero of both $P_1(z)$ and $P_2(z)$ by Lemma \ref{abov}. By Remark \ref{in2}, $z_0$ is a real critical point of $f(z)$. Moreover, $z_0 \in {\Gamma}$ by \eqref{eq:trcurve}. 
Theorem \ref{TTH} implies that in the neighbourhood of $z_0$,
 \begin{eqnarray*}
 f(z)= \frac{B^k(z)}{A(z)}=(z-z_0)^{pk}q(z)	
 \end{eqnarray*} 
 where $q(z)$ is analytic at $z_0$ and  $q(z_0)\neq 0$.
Pick a $\delta_1 >0$ such that  $q(z)$ is non-vanishing in $D_{\delta_1}(z_0 )$. 
In this neighbourhood, there exists an analytic function  $q_1(z)$ such that $q_1(z)=\sqrt[pk]{q(z)}$. Take $q_1(z)$ as a branch of $\sqrt[pk]{q(z)}$.  Define
$u(z):=(z-z_0)q_1(z)$. Then we have
\begin{eqnarray*}
 f(z)=u(z)^{pk},~~~~\mbox{where}~~~~ u(z_0)=0,~~ q_1(z_0)=u'(z_0)\neq 0.
\end{eqnarray*}
Thus for a small positive $\epsilon \in \mathbb{R}$, the equation 
\begin{eqnarray*}
 f(z)=\pm \varepsilon
\end{eqnarray*}
is equivalent to 
\begin{eqnarray}\label{cbcbi}
(z-z_0)^{pk} q_1(z)^{pk}=\pm \varepsilon. \end{eqnarray}
 Let $h(z)$ be the inverse function to $u(z)$. Then by
applying Theorem \ref{TTH} to \eqref{cbcbi} where the left side of this equation has a zero $z=z_0$ with multiplicity $pk$, we obtain solutions of the form  
\begin{eqnarray*}
z=z_0 + h(\varepsilon^{(1/pk)} e^{(2 \pi i j/ pk)}) \end{eqnarray*}
or
\begin{eqnarray*}
z=z_0 + h(\varepsilon^{(1/pk)} e^{(i\pi + 2 \pi i j)/ pk)})\end{eqnarray*}
where  $\varepsilon^{(1/pk)}$ is the $pk$-th roots of $\varepsilon.$
Using the fact that $pk >2$ and $h$ has a simple zero at $0$, we deduce that \eqref{cbcbi}  has $pk$ solutions $z$ and these cannot be all real.

Denote by $\rho= \frac{k^k}{(k-1)^{k-1}}$. Then, by Theorem 2 [7],
the zeros of $P_n(z)$ are contained in $\Gamma = f^{-1}([0, \rho] )$ or 
$\Gamma = f^{-1}([-\rho, 0] )$ when $k$ is even or odd respectively, and these zeros are dense on $\Gamma$ as $n \to \infty.$ Now for $[0, \varepsilon]\subset [0, \rho]$ and $ [-\varepsilon, 0]\subset [-\rho, 0],$ it follows that  $f^{-1}([0, \varepsilon]) \subset \Gamma$  or $f^{-1}([ -\varepsilon, 0]) \subset \Gamma$. For all the polynomials $P_n(z)$ to be hyperbolic, we require $\Gamma$ to consist only of intervals on the real line in $\mathbb{C}$. From the  solutions of \eqref{cbcbi}, it is clear that neither $f^{-1}([0, \varepsilon])$ nor $f^{-1}([ -\varepsilon, 0])$  is a subset of only real intervals. Thus there will always be at least one non-real curve through $z_0$ on which non-real zeros of $P_n(z)$ will be located. The conclusion follows.
\end{proof}

\section{Examples}\label{BIII}
In this section we present concrete examples using numerical experiments. In these examples, we consider the sequence of polynomials $\{P_n(z)\}_{n=0}^{\infty}$ generated by the rational function 
\begin{equation}\label{uuuu}
 \sum_{n=0}^{\infty}P_n(z)t^n= \frac{1}{1+B(z)t+ A(z)t^k},
 \end{equation}
where $A(z)$ and $B(z)$ are coprime real polynomials. We plot a graph showing
\begin{enumerate}[(i)]
\item a portion (black curves) of the curve $\widetilde{\Gamma}$  given by $\Im \left(\frac{B^k(z)}{A(z)}\right)=0$; 
\item the zeros (circles) of one of the polynomials  $P_n(z)$  in \eqref{uuuu} (of our choice) described by specifying  $k, A(z), B(z)$ and a positive integer $n$. These are located on $\Gamma$;
\item the points $z^*\in \mathbb{C}$ which are endpoints of the curve $\Gamma$ (indicated by black dots). Such points are the elements of the set $E$ defined in Lemma \ref{Namak}.  
\end{enumerate}
 
\begin{example} For $n=71,k=3, A(z)=z^3 - z^2 - 5 z + 7$ and  $B(z)=z^2 - z -6$, we obtain Fig. \ref{a1}.
\begin{figure}[h]\begin{center}
$ 
\begin{array}{c}
\includegraphics[height=8cm, width=12cm]{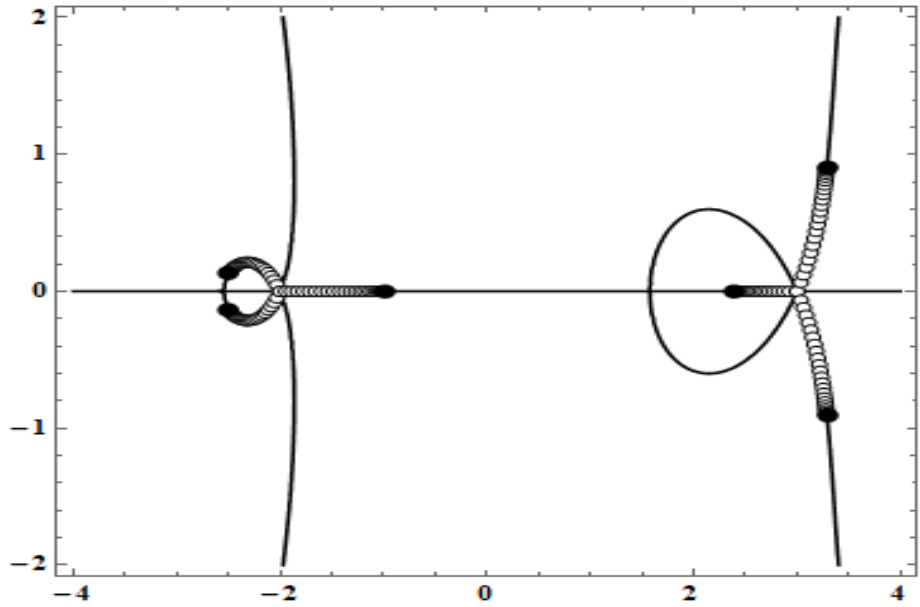}
\end{array}
$
\end{center}
\caption{The rational function is given by $f(z)=1/(1+(z^2 - z-6)t+(z^3 -z^2-5z+7 )t^3)$.}\label{a1}
\end{figure} 
\end{example}

\begin{example} For $n=150, k=5, A(z)=z^2+z-4$ and  $B(z)=z^2+z-2$, we obtain Fig. \ref{a15}.  
\begin{figure}[h]\begin{center}
$ 
\begin{array}{c}
\includegraphics[height=8cm, width=12.3cm]{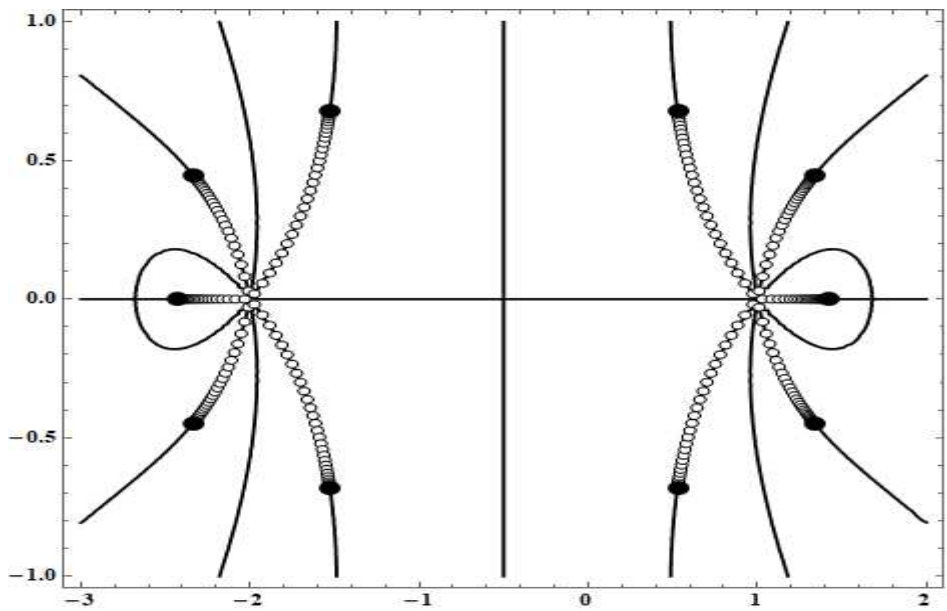}
\end{array}
$
\end{center}
\caption{The rational function is given by $f(z)=1/(1+(z^2+z-2)t+(z^2 + z -4 )t^5)$.}\label{a15}

\end{figure} 
\end{example}
\section{Final Remarks}
Problem 1 has been settled in the negative in the sense that it is not possible to generate a sequence of hyperbolic polynomials using the recurrence  (\ref{christ})  with the given initial conditions \eqref{kachap}.

\medskip

{\bf Acknowledgements.}  I am indebted to my advisor Professor Boris Shapiro for the discussions and comments. I am thankful  Dr. Alex Samuel Bamunoba, the fruitful discussions and guidance. Thanks go to the anonymous referee for careful reading this work, and giving  insightful comments and suggestions. I also acknowledge the financial support provided by Sida Phase-IV bilateral program with Makerere University.

\end{document}